\documentclass[12pt,letterpaper]{article}
\usepackage{blindtext}
\usepackage[T1]{fontenc}
\usepackage{inputenc}
\usepackage{amsmath,amsfonts, amssymb, amsthm}
\usepackage{mathtools}
\usepackage{multirow}
\usepackage{tocloft}
\usepackage{hyperref}
\usepackage{float}
\usepackage{dirtytalk}
\usepackage{accents}
\usepackage{scalerel}
\usepackage{comment}
\usepackage{enumitem}
\usepackage{tcolorbox}
\usepackage{graphicx}
\usepackage{subfig}
\usepackage{tikz}
\usepackage{tikz-cd}
\usepackage{cleveref}
\usepackage[hmargin=2cm]{geometry}
\usepackage[backend=biber, style=numeric, maxnames=50]{biblatex}
\definecolor{lgreen} {RGB}{180,210,100}
\definecolor{dblue}  {RGB}{20,66,129}
\definecolor{ddblue} {RGB}{11,36,69}
\definecolor{lred}   {RGB}{220,0,0}
\definecolor{nred}   {RGB}{224,0,0}
\definecolor{norange}{RGB}{230,120,20}
\definecolor{nyellow}{RGB}{255,221,0}
\definecolor{ngreen} {RGB}{98,158,31}
\definecolor{dgreen} {RGB}{78,138,21}
\definecolor{nblue}  {RGB}{28,130,185}
\definecolor{jblue}  {RGB}{20,50,100}

\usepackage{biblatex} 
\usetikzlibrary{shapes, fit}
\usetikzlibrary{shapes, arrows.meta, positioning}
\usetikzlibrary{arrows.meta}
\usetikzlibrary{shapes.geometric}
\usetikzlibrary{matrix,arrows.meta}

\graphicspath{ {./images/} }

\DeclareMathOperator{\tr}{tr}

\newcommand{\N}{\mathbb{N}}
\newcommand{\Z}{\mathbb{Z}}

\newcommand{\cut}{\setminus}
\newcommand{\abs}[1]{\left\lvert #1 \right\rvert}

\newcommand\thickbar[1]{\accentset{\rule{.5em}{1.0pt}}{#1}}
\newcommand\thickoverline[1]{\accentset{\rule{.8em}{1.0pt}}{#1}}
\newcommand\extrafootertext[1]{%
    \bgroup
    \renewcommand\thefootnote{\fnsymbol{footnote}}%
    \renewcommand\thempfootnote{\fnsymbol{mpfootnote}}%
    \footnotetext[0]{#1}%
    \egroup
}

\def\msquare{\mathord{\scalerel*{\Box}{gX}}}

\newtheorem{theorem}{Theorem}[section]
\newtheorem{conjecture}{Conjecture}[section]

\newtheorem{corollary}{Corollary}[theorem]
\newtheorem{lemma}[theorem]{Lemma}
\newtheorem{definition}[theorem]{Definition}
\newtheorem*{remark}{Remark}

\title{On the Generation, Structure, and Symmetries of Minimal Prime Graphs}
\author{
\textsc{Ziyu Huang}\\
  \normalsize{\emph{Department of Mathematics, Boston College}}\\
  \normalsize{\emph{Boston, MA, U.S.A.}}\\
\normalsize{\texttt{huangaaf@bc.edu}}
  \\
  \\
  \textsc{Thomas Michael Keller}\\
   \normalsize{\emph{Department of Mathematics, Texas State University}} \\
   \normalsize{\emph{San Marcos, CA, U.S.A.}} \\
  \normalsize{\texttt{keller@txstate.edu}}
  \\
  \\
  \textsc{Shane Kissinger}\\
  \normalsize{\emph{Department of Mathematics, Harvard University}} \\
  \normalsize{\emph{Cambridge, MA, U.S.A.}}\\
\normalsize{\texttt{skissinger@college.harvard.edu}}
  \\
  \\
  \textsc{Wen Plotnick}\\
  \normalsize{\emph{Department of Mathematics, University of Michigan}} \\
  \normalsize{\emph{Ann Arbor, MI, U.S.A.}}\\
  \normalsize{\texttt{plotnw@umich.edu}}
  \\
  \\
  \textsc{Maya Roma}\\
  \normalsize{\emph{Department of Mathematics, University of California-Berkeley}}\\
  \normalsize{\emph{Berkeley, CA, U.S.A.}} \\
  \normalsize{\texttt{mayaroma@berkeley.edu}}
}
\date{}
  
\begin{document}
  
\maketitle
  
\begin{abstract}
In this paper we continue the study of prime graphs of finite solvable groups. The prime graph, or Gruenberg-Kegel graph, of a finite group G has vertices consisting of the prime divisors of the order of G and an edge from primes p to q if and only if G contains an element of order pq. Since the discovery of a simple, purely graph theoretical characterization of the prime graphs of solvable groups in 2015 these graphs have been studied in more detail from a graph theoretic angle. In this paper we explore several new aspects of these graphs. We characterize regular reseminant graphs and study the automorphisms of reseminant graphs for arbitrary base graphs. We then study minimal prime graphs on larger vertex sets by a novel regular graph construction for base graphs and by proving results on prime graph properties under graph products. Lastly, we present the first new way, different from vertex duplication, to obtain a new minimal prime graph from a given minimal prime graph.

\end{abstract}


\section{Introduction}
\extrafootertext{2010 Mathematics Subject Classification. Primary: 20D10, Secondary: 05C25.}\extrafootertext{Key words and phrases. prime graph, solvable group, 3-colorable, triangle-free, automorphism, graph product, regular graph.}
The \emph{prime graph} of a finite group $G$, also known as a Gruenberg-Kegel graph, is the graph with vertex set defined as the primes dividing $|G|$ and edges $\{p, q\}$ for primes $p$ and $q$ if and only if there is an element of order $pq$ in $G$. A recent and powerful result \cite{gruber} is that a graph is isomorphic to the prime graph of a finite solvable group if and only if its complement is triangle-free and $3$-colorable. 
As we will use the idea of a prime graph of a finite solvable group heavily, we call these \emph{solvable prime graphs}. A particularly important subclass of solvable prime graphs is that of \emph{minimal prime graphs}. Minimal prime graphs have the least possible edges such that removing an edge causes the complement to no longer be 3-colorable or triangle-free.

In this paper we continue the study of solvable prime graphs begun in \cite{gruber} and \cite{florez}. We will present a multitude of new results on these graphs,
thereby further developing some leads from the earlier work.\\
One of the first discoveries on the subject was that the process of \emph{vertex duplication} is a tool to create a new minimal prime graph from a given one by adding a new vertex. Since then, \emph{reseminant graphs} --- minimal prime graphs obtained by repeated vertex duplication from the 5-cycle $C_5$ (the smallest minimal prime graph) --- have been actively studied. 

Results in \cite{florez} aim to generalize the notion of reseminant graphs, which we expand in this paper in Section 2. In that section we also take the study of reseminant graphs into new directions by characterizing regular reseminant graphs and the structure of automorphism groups of reseminant graphs for a starting graph that is not necessarily $C_5$. From this investigation, a natural need arises for increasingly complex examples of minimal prime graphs. This led us to our next avenue of study. 

We wanted to find examples of minimal primes graphs on increasingly larger sets of vertices which cannot be  generated from vertex duplication. (We call such graphs base graphs.)  Using results from \cite{sidorenko}, in Section 3 we introduce a new method to produce an infinite family of base graphs, which includes $C_5$. In Section 4 we then extend the search for minimal prime graphs on larger vertex sets by investigating the effect of products like the direct product, Cartesian product, and strong product in preserving minimality. 

Lastly, there has been a search for other generation methods like vertex duplication to produce new minimal prime graphs from given ones. In Section 5 we present the first advance in this direction: a method we call {\it clique generation}. This is the first method that is separate from vertex duplication.

All code used to create and validate solvable prime graphs and minimal prime graphs can be found at \cite{githubrepository} and was created using \cite{sage}.

\subsection{Notation and Background}
\begin{itemize}
    \item All graphs are simple and undirected unless otherwise stated. We define a graph $\Gamma$ to be the pairing $\Gamma = (V(\Gamma), E(\Gamma))$ where $V(G)$ and $E(G)$ are the vertex and edge sets, respectively.
    \item Given a graph $\Gamma$, the graph $\overline{\Gamma}$ is its complement. 
    \item We denote the graph $\Gamma\setminus \{u,v\}$ to be the graph with the vertices $p$ and $q$ and their associated edges removed. 
    \item We use the notation $\lbrace u, v\rbrace$ to denote an undirected edge between vertices $u$ and $v$.
    \item We use $N_1(v)$ to denote the neighbors of $v$ i.e. the set of vertices with distance $1$ from $v$. We use $N_1[v]$ to denote the closed neighborhood of $v$, i.e. set of vertices with distance less than or equal to $1$ from $v$.
    \item Two vertices $u$ and $v$ are called \emph{true twins} if they share the same closed neighborhood, i.e. if $N_1[u] = N_1[v]$.
    \item A set of vertices $K$ is called a \emph{clique} if all vertices in $K$ are adjacent.
    \item The set of reseminant graphs $\mathcal{R}$ are the graphs generated by vertex duplication from the Cycle Graph of 5 vertices $C_5$.
    \item For any graph $\Gamma$, we define the set of $\Gamma$-reseminant graphs to be the set of those graphs which can be obtained by starting with $\Gamma$ and performing a finite number of vertex duplications on it. 
    \item All groups are finite unless otherwise stated.
\end{itemize}

\begin{definition}[Vertex Duplication]
    \label{def:Vertex Duplication}
    Suppose $\Gamma$ is a minimal prime graph with a subset of vertices $U$ such that there exists a vertex $w\in\Gamma$ with $V=N_1[w]$. Then, the graph
    \begin{equation*}
        \Gamma^{\prime}=(V(\Gamma)\cup\lbrace w^{\prime}\rbrace, E(\Gamma)\cup\lbrace \lbrace w', u\rbrace\mid u\in U\rbrace)
    \end{equation*}
    i.e. the graph formed by adding a vertex $w^\prime$ adjacent to the same vertices as $w$ as well as $w$ itself, is a minimal prime graph generated from $\Gamma$.
\end{definition}

\begin{definition}
    A minimal prime graph $\Gamma$ is a connected graph on two or more vertices such that
\begin{enumerate}[label=(\arabic*)]
\item $\overline{\Gamma}$ is triangle-free
\item $\overline{\Gamma}$ is $3$-colorable
\item For any edge $\{u,v\}$ in $\Gamma$, the graph $\overline{\Gamma \setminus \{u,v\}}$ is no longer triangle-free and $3$-colorable
\end{enumerate}
\end{definition}

A simple example of a minimal prime graph is below.  
\begin{figure}[H]
\centering
\begin{minipage}{.49\textwidth}
\begin{center}
\begin{tikzpicture}
\begin{scope}[every node/.style={circle,thick,draw}]
    \node (8) at ( 1.41421356237310 , 1.41421356237310 ) {8};
    \node (1) at ( 0.000000000000000 , 2.00000000000000 ) {1};
    \node (2) at ( -1.41421356237310 , 1.41421356237310 ) {2};
    \node (3) at ( -2.00000000000000 , 0.000000000000000 ) {3};
    \node (4) at ( -1.41421356237310 , -1.41421356237310 ) {4};
    \node (5) at (0,-2) {5};
    \node (6) at ( 1.41421356237310 , -1.41421356237310 ) {6};
	\node (7) at (2,0) {7};
\end{scope}
\begin{scope}[>={Stealth[black]},
              every edge/.style={draw=black,very thick}]
           \path [-] (1) edge node {} (8);   
           \path [-] (1) edge node {} (2); 
           \path [-] (1) edge node {} (3); 
           \path [-] (1) edge node {} (7);  
           
           \path [-] (2) edge node {} (3);   
           \path [-] (2) edge node {} (4); 
           \path [-] (2) edge[style={draw=nred, very thick}] node {} (8); 
             
           \path [-] (3) edge node {} (4);   
           \path [-] (3) edge node {} (5); 
           
           \path [-] (4) edge node {} (5);   
           \path [-] (4) edge node {} (6); 

		   \path [-] (5) edge node {} (6);   
           \path [-] (5) edge node {} (7); 
            
           \path [-] (6) edge node {} (7);   
           \path [-] (6) edge node {} (8);
           
           \path [-] (7) edge node {} (8);   
           \path [-] (6) edge node {} (8);
\end{scope}
\end{tikzpicture}
\end{center}
\end{minipage}
\begin{minipage}{.5\textwidth}
\begin{center}
\begin{tikzpicture} 
\begin{scope}[every node/.style={circle,thick,draw}]
    \node (8) at ( 1.41421356237310 , 1.41421356237310 ) {8};
    \node (1) at ( 0.000000000000000 , 2.00000000000000 ) {1};
    \node (2) at ( -1.41421356237310 , 1.41421356237310 ) {2};
    \node (3) at ( -2.00000000000000 , 0.000000000000000 ) {3};
    \node (4) at ( -1.41421356237310 , -1.41421356237310 ) {4};
    \node (5) at (0,-2) {5};
    \node (6) at ( 1.41421356237310 , -1.41421356237310 ) {6};
	\node (7) at (2,0) {7};
\end{scope}
\begin{scope}[
              every edge/.style={draw=black,very thick}]
           \path [-] (1) edge node {} (4);   
           \path [-] (1) edge node {} (5); 
           \path [-] (1) edge node {} (6);

           \path [-] (2) edge[style={draw=nblue, very thick}] node {} (5);   
           \path [-] (2) edge node {} (6); 
           \path [-] (2) edge node {} (7); 
           \path [-] (2) edge[style={draw=nred, very thick, dashed}] node {} (8); 
           \path [-] (3) edge node {} (8);   
           \path [-] (3) edge node {} (7); 
           \path [-] (3) edge node {} (6); 
           
           \path [-] (4) edge node {} (1);   
           \path [-] (4) edge node {} (8); 
           \path [-] (4) edge node {} (7); 
           
           \path [-] (5) edge node {} (1); 
           \path [-] (5) edge node {} (8); 
           
           \path [-] (6) edge node {} (3);   
           \path [-] (6) edge node {} (2); 
           \path [-] (6) edge node {} (1); 
           
           \path [-] (7) edge node {} (4);   
           \path [-] (7) edge node {} (3); 
           \path [-] (7) edge node {} (2); 
           
           \path [-] (8) edge node {} (3);   
           \path [-] (8) edge node {} (4); 
           \path [-] (8) edge[style={draw=nblue, very thick}] node {} (5); 
\end{scope}
\end{tikzpicture}
\end{center}
\end{minipage}
\caption{Removing the red edge from the minimal prime graph on the left results in a triangle in its complement on the right.}
\end{figure}
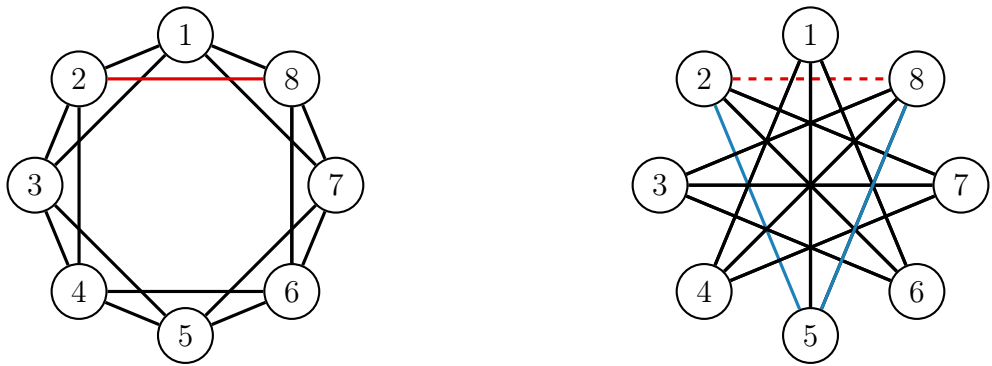


\section{Products and Automorphisms of \texorpdfstring{$\Gamma$}{Gamma}-Reseminant Graphs}

Reseminant graphs in \cite{gruber} originally referred to the family of graphs generated from repeated vertex duplication on the cycle graph $C_5$. Reseminant graphs were then generalized in \cite{florez} to $\Gamma$-reseminant graphs, where $\Gamma$ is a base graph. We continue to investigate $\Gamma$-reseminant graphs by developing techniques to determine graph regularity and by formalizing the structure of their automorphism groups. 

\begin{definition}
\label{def:BaseGraph}
A graph is a base graph if no two vertices are adjacent with the same adjacency relations to the other vertices, i.e. there are no sets of true twins. 
\end{definition}

\subsection{\texorpdfstring{$C_5$}{C\_5}-Reseminant, Direct Products, and Vertex Duplication as a Matrix}
In our investigations of $\Gamma$-reseminant graphs, we found it useful to model vertex duplication in the language of matrices. We used these matrices to determine when a graph would be regular under repeated vertex duplications. The following results easily expand to arbitrary $\Gamma$-reseminant graphs. But, we will consider the case of $\Gamma = C_5$ and motivate the results with a short example. 
\begin{figure}[H]
\centering
\begin{tikzpicture}[scale=.75,
dot/.style = {circle,draw, inner sep=0, minimum size=1.5em},
dot/.default = 6pt  
]
\begin{scope}[every node/.style={dot}]
    \node (5) at (0,.2) {4};
    \node (4) at (1.5,1.3) {3};
    \node (3) at (3,.2) {2};
    \node (2) at (2.2,-1.3) {1};
    \node (1) at (0.7,-1.3) {0};
     \node (6) at (-.5,-2.4) {5};
\end{scope}

\begin{scope}[>={Stealth[black]},
              every edge/.style={draw=black,very thick}]
    \path [-] (1) edge node {} (2);
     \path [-] (2) edge node {} (3);
     \path [-] (3) edge node {} (4);
     \path [-] (4) edge node {} (5);
     \path [-] (5) edge node {} (1);
     \path [-] (6) edge[style={draw=nred,very thick}] node {} (1);
     \path [-] (6) edge[style={draw=nred,very thick}] node {} (2);
     \path [-] (6) edge[style={draw=nred,very thick}] node {} (5);
\end{scope}
\end{tikzpicture}
\caption{Vertex 0 duplicated once on the Cycle Graph $C_5$.}
\label{fig:vertexdup}
\end{figure}
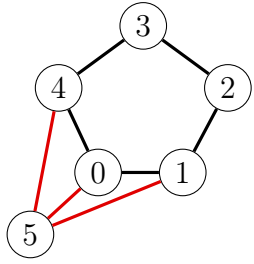
Duplicating a vertex on $C_5$, increases the degree of the original and its neighbors by 1. Likewise, duplications on vertices 4 and 1 increases the degree of the set of true twins of vertex 0 by 2. The adjacency matrix of $C_5$ identifies the number of neighbors a vertex has, so adding the identity matrix gives us a way to model vertex duplication.

As it is not intuitively obvious, we observe that a minimal prime graph does not need to be regular. As a counterexample, take the non-regular minimal prime graph below. 
\begin{figure}[H]
\centering
\begin{tikzpicture}[scale=.8,
dot/.style = {circle,draw, inner sep=0, minimum size=1.5em},
dot/.default = 6pt  
]
\begin{scope}[every node/.style={dot}]
    \node (0) at (3,0) {0};
    \node (1) at (2.523,1.6219) {1};
    \node (2) at (1.2462,2.7288) {2};
    \node (3) at (-0.42694,2.9694) {3};
    \node (4) at (-1.9645,2.2672) {4};
    \node (5) at (-2.8784,0.845197) {5};
    \node (6) at (-2.8784,-0.845197) {6};
    \node (7) at (-1.9645,-2.2672) {7};
    \node (8) at (-0.4269,-2.9694) {8};
    \node (9) at (1.2462,-2.72889) {9};
    \node (10) at (2.5237,-1.62192) {10};
\end{scope}

\begin{scope}[>={Stealth[black]},
              every edge/.style={draw=black,very thick}]
    \path [-] (0) edge node {} (2);
    \path [-] (0) edge node {} (4);
    \path [-] (0) edge node {} (5);
    \path [-] (0) edge node {} (6);
    \path [-] (0) edge node {} (7);
    \path [-] (0) edge node {} (8);
    \path [-] (0) edge node {} (9);
    \path [-] (0) edge node {} (10);
    
    \path [-] (1) edge node {} (3);
    \path [-] (1) edge node {} (6);
    \path [-] (1) edge node {} (8);
    \path [-] (1) edge node {} (9);
    \path [-] (1) edge node {} (10);

    \path [-] (2) edge node {} (0);
    \path [-] (2) edge node {} (4);
    \path [-] (2) edge node {} (5);
    \path [-] (2) edge node {} (6);
    \path [-] (2) edge node {} (7);
    \path [-] (2) edge node {} (8);
    \path [-] (2) edge node {} (9);
    \path [-] (2) edge node {} (10);
    
    \path [-] (3) edge node {} (1);
    \path [-] (3) edge node {} (7);
    
    \path [-] (4) edge node {} (0);
    \path [-] (4) edge node {} (2);
    \path [-] (4) edge node {} (5);
    \path [-] (4) edge node {} (6);
    \path [-] (4) edge node {} (7);
    \path [-] (4) edge node {} (8);
    \path [-] (4) edge node {} (9);
    \path [-] (4) edge node {} (10);
    
    \path [-] (5) edge node {} (0);
    \path [-] (5) edge node {} (2);
    \path [-] (5) edge node {} (4);
    \path [-] (5) edge node {} (6);
    \path [-] (5) edge node {} (7);
    \path [-] (5) edge node {} (8);
    \path [-] (5) edge node {} (9);
    \path [-] (5) edge node {} (10);
    
    \path [-] (6) edge node {} (0);
    \path [-] (6) edge node {} (1);
    \path [-] (6) edge node {} (2);
    \path [-] (6) edge node {} (4);
    \path [-] (6) edge node {} (5);
    \path [-] (6) edge node {} (8);
    \path [-] (6) edge node {} (9);
    \path [-] (6) edge node {} (10);
     
    \path [-] (8) edge node {} (0);
    \path [-] (8) edge node {} (1);
    \path [-] (8) edge node {} (2);
    \path [-] (8) edge node {} (4);
    \path [-] (8) edge node {} (5);
    \path [-] (8) edge node {} (6);
    \path [-] (8) edge node {} (9);
    \path [-] (8) edge node {} (10);
    
    \path [-] (9) edge node {} (0);
    \path [-] (9) edge node {} (1);
    \path [-] (9) edge node {} (2);
    \path [-] (9) edge node {} (4);
    \path [-] (9) edge node {} (5);
    \path [-] (9) edge node {} (6);
    \path [-] (9) edge node {} (8);
    \path [-] (9) edge node {} (10);
    
\end{scope}
\end{tikzpicture}
    \caption{Non-regular minimal prime graph.}
    \label{fig:nonregmpg}
\end{figure}
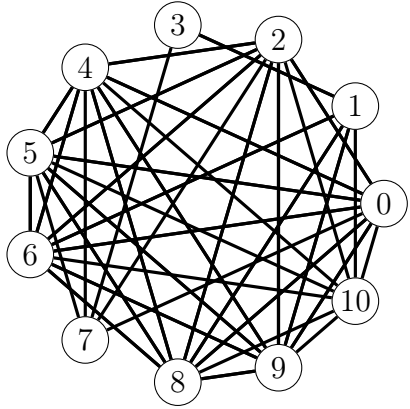

We demonstrate the usefulness of the reformulation of vertex duplication in terms of matrices by proving that $C_5$-reseminant graphs are only regular of a certain form. Furthermore, we prove that there are no non-trivial $C_5$-reseminant graphs that are complement direct products. The \emph{direct (or tensor or Kronecker) product}  $G\times H$ of graphs $G$ and $H$ is defined as the graph with vertex set $V(G)\times V(H)$ and $\{(x,y),(x',y')\}$ is an edge in $G\times H$ if and only if $\{x,x'\}\in E(G)$ and $\{y,y'\}\in E(H)$.

Note that the direct product of regular graphs is a regular graph that has a degree equal to the product of the degrees of its factors. As we will use it later on, note that if $G$ is $k$-regular, then the compliment $\overline{G}$ is a $(n-k-1)$-regular graph.

\begin{lemma}\label{lemma:regular}
	If $\Gamma_1, \dots, \Gamma_n$ are graphs such that $\Gamma_j$ is $k_j$-regular for $j\in[1,n]$, then $\displaystyle \bigtimes_{i=1}^n \Gamma_i$ is regular of degree $\prod_{j=1}^nk_j$. 
\end{lemma}

\begin{proof}
	Suppose $\Gamma_j$ is a $k_j$-regular graph with $m_j$ vertices, then $A(\Gamma_j)\cdot{\bf 1}_j = k_j\cdot {\bf 1}_j$ where ${\bf 1}_j$ is the all-1 column vector of length $m_j$ for $j\in[1,n]$. Consider the adjacency matrix $A(\bigtimes_{i=1}^n \Gamma_i) = \bigtimes_{i=1}^n A(\Gamma_i),$ which is a $\left(\prod_{j=1}^n m_j\right)\times \left(\prod_{j=1}^n m_j\right)$ matrix. It follows that the all-1 column vector of length $\prod_{j=1}^n m_j$ is equal to $\bigtimes_{j=1}^n {\bf 1}_j$ and so we have:
	
	$$
	  \left(\bigtimes_{j=1}^n A(\Gamma_j)\right) \cdot \bigtimes_{j=1}^n {\bf 1}_j = \bigtimes_{j=1}^n A(\Gamma_j)\cdot{\bf 1}_j= \bigtimes_{i=1}^n k_j\cdot{\bf 1}_j =  \left(\prod_{j=1}^nk_j\right)\bigotimes_{i=1}^n {\bf 1}_j  	$$
	
	Therefore, the direct product $\bigtimes_{i=1}^n \Gamma_i$ is regular of degree $\prod_{j=1}^nk_j$. 
\end{proof}

\begin{theorem}\label{thm:resnonregular}
	All graphs generated from $C_5$ through vertex duplication in $\mathcal{R}$ are non-regular or $k$-regular on $n$ vertices where $k=2+3h$ and $n=5+5h$ for some positive integer $h$. 
\end{theorem}

\begin{proof}
   Let $G\in\mathcal{R}$ be an arbitrary graph generated from vertex duplication on $C_5$. We choose a labeling of $G$ such that an isomorphic copy of $C_5$ is labeled $1,2,\dots, 5$. Let $\Pi$ be a partition of $V(G)$ where $\Pi = \{V_1, V_2,\dots, V_5\}$ such that $V_i$ is the set of true twins of vertex $i\in\{1,2, \dots, 5\}$. By Lemma 7.4 of \cite{florez} vertex duplication on any vertex in the same set of true twins $V_i\in\Pi$ produces isomorphic graphs. Thus, we can treat our graph $G$ as having been generated from some sequence of duplications of the vertices $1, 2, \dots, 5$. Denote this sequence by $(a_1, \dots, a_d)$, where $d$ denotes the total number of vertex duplications and $a_i\in \{1, 2,\dots, 5\}$ for all $i$.
   
   Next, we show that we only need consider the number of times each vertex is duplicated and can safely ignore the order in which this is done. Let $G'$ be the reseminant graph that is given by a reordering of the sequence $(a_i)$. We define the identity isomorphism between the base graphs of $G$ and $G'$, which are both $C_5$. By Lemma 7.3 of \cite{florez},
   this extends to an isomorphism $\varphi: G \to G'$ since the set of true twins of $G$ and $G'$ are identical. So, we can safely reorder the sequence. Therefore, we need only to consider the multiplicities of each of the 5 vertices. 
   
   Let $v = (v_1, v_2, v_3, v_4, v_5)^T$ be the vector whose component $v_i$ for $i\in\{1, \dots, 5\}$ represents the degree of the vertices in the set $V_i$. Let $w = (w_1, w_2, w_3, w_4, w_5)^T$ be the vector where $w_i$ equals the number of times a vertex was duplicated in $V_i$ to attain the graph $G$. The vector $w$ uniquely determines the graph $G$ as any tuple with multiplicities $w_i$ produces a graph isomorphic to $G$. This follows directly from Lemma 7.5 of \cite{florez}. As vertex duplication creates a vertex with the same adjacency relations, we can model the action of vertex duplication via the following matrix.  \begin{equation*}
        A=\begin{bmatrix}1 & 1 & 0 & 0 & 1\\
             1 & 1 & 1 & 0 & 0\\
             0 & 1 & 1 & 1 & 0\\
             0 & 0 & 1 & 1 & 0\\
             1 & 0 & 0 & 1 & 1\end{bmatrix}
    \end{equation*}

We can then relate the vectors $w$ and $v$ by the relation $Aw+(2,2,2,2,2)^T=v$. Note the all-2 column is from the starting degree of $C_5$. As we are determining exactly when $G$ is $k$-regular for some positive integer $k$, we set $v=(k, k, k, k, k)$ and get $w = (\frac{k-2}{3},\frac{k-2}{3},\frac{k-2}{3},\frac{k-2}{3},\frac{k-2}{3})^T$. The determinant of $A$ is 3, so we conclude $w$ is the unique solution. As we can only duplicate vertices an integer number of times, we restrict our attention to the case when $w$ has integral entries. For this to be the case, we see we must have $k=2+3h$.

We conclude that if $G$ is $k$-regular, $k=2+3h$ for some $h$. We have a total of $5h$ vertex duplications, so the total number of vertices in $G$ is $5+5h$. It follows then that $G$ is non-regular for all other integer solutions to $A\vec{w}+2I = \vec{v}$. 
\end{proof}

We define the complementary direct product $\thickbar \times$ as $G\thickbar{\times} H = \overline{\overline{G} \times \overline{H}}$. As an application of \Cref{thm:resnonregular}, we will show that the only reseminant graph contained within $T = \{\thickoverline{\bigtimes}_{i=1}^n C_5\,|\,n\in \N\}$ is $C_5$. We will show later that the graphs contained within $T$ are solvable prime graphs. 

\begin{theorem}\label{thm:intersection}
	Let $T = \{\thickoverline{\bigtimes}_{i=1}^n C_5\,|\,n\in \N\}$ be the prime graphs constructed from repeated direct products. Then, the sets $T$ and $\mathcal{R}$ intersect at $T \cap \mathcal{R} = \{C_5\}$.
\end{theorem}

\begin{proof}
	Suppose $\Gamma \in T\cap\mathcal{R}$ and $G$ is not isomorphic to $C_5$. Since $\Gamma$ is contained in $T$, we have that $\Gamma = \thickoverline{\bigtimes}_{i=1}^t C_5$ for some positive integer $t$ and so $\Gamma$ has $n=5^t$ total vertices. Note that since $C_5$ is self-complementary, we have that $\thickoverline{\bigtimes}_{i=1}^t C_5 = \overline{\bigtimes_{i=1}^t C_5}$.  By \Cref{lemma:regular}, the graph $\bigtimes_{i=1}^t C_5$ is $2^t$ regular and so $\Gamma$ is $(5^t-2^t-1)$-regular.

	But, $\Gamma$ is also contained in $\mathcal{R}$ and is regular. By \Cref{thm:resnonregular}, there exists a positive integer $h$ such that $5^t-2^t-1 = 2+3h$ and $5^t=5+5h$. 
	
Substituting for $h = 5^{t-1}-1$, we get
\[
\begin{split}
	5^t - 2^t - 1 & = 2+3(5^{t-1}-1) \\
	5^t - 3\cdot \frac{5^t}{5} - 2^t & = 0 \\
	\frac{2}{5}5^t - 2^t & = 0 \\
	2(5^{t-1} - 2^{t-1}) & = 0
\end{split}
\]

It follows that our only solution is $t=1$. But, this means that $\Gamma$ was actually $C_5$ and so $T\cap\mathcal{R} = \{C_5\}$. 	
\end{proof}

\subsection{Automorphism Group Structure of \texorpdfstring{$\Gamma$}{Gamma}-Reseminant Graphs}

From observation, the structure of a $\Gamma$-reseminant graph appears to follow the symmetry of the base graph $\Gamma$. This section examines the connection between the symmetries of a $\Gamma$-reseminant graph and $\Gamma$ through their automorphism groups.

\begin{lemma}\label{lemma:completegraph}
	Let $\Gamma = (V,E)$ and $H$ be a set of true twins of a vertex $v$ with size $n$. Then, the induced subgraph $\Gamma[H] = (H, E_H)$ where $E_H = \{\{u,v\}\,|\,u,v\in H, \{u,v\}\in E\}$ is isomorphic to $K_n$.
\end{lemma}

\begin{proof}
	Recall that vertex duplication creates a new vertex that has identical adjacency relations to the original. The claim follows immediately.
\end{proof}

\begin{theorem}\label{thm:normalsubgroup}
	Given a base graph $\Gamma$ on $n$ vertices, let $G$ be a $\Gamma$-reseminant graph with sets of true twins $V_i$ of size $h_i$ for $i\in\{1,2,\dots, n\}$. Then, there exists a subgroup $H$ isomorphic to the direct product of symmetric groups $S_{h_1}\times \cdots \times S_{h_n}$ such that $H\lhd \mathrm{Aut}(G)$.
	
	Furthermore, $\mathrm{Aut}\,G[V_1]\times \cdots \times \mathrm{Aut}\,G[V_n]\cong H$. 
\end{theorem}

\begin{proof}
	Let $\Pi$ be a partition of $V(G)$ into the set of true twins $V_1, V_2, \dots, V_n$ where $V_i$ is the set of the vertex $i\in \{1,2,\dots, n\}$ in the base graph $\Gamma$. Let $\Pi$ also be an $\mathrm{Aut}\,G$-set where $g\in \mathrm{Aut}\,G$ acts on the $\Pi$ by $gV_i = V_{\sigma(i)}$ for a permutation $\sigma\in S_n$. We have then a homomorphism $\psi:\mathrm{Aut}\,G \to S_n$ defined by sending an automorphism $g$ to its induced permutation $\sigma\in S_n$. 
	
	Let $H$ be the kernel of this homomorphism. Then $H$ is a normal subgroup of $\mathrm{Aut}\,G$ and consists of all automorphisms of $G$ which act trivially on $\Pi$, but permute the individual sets $V_i$ of true twins internally
	
	Recall that by \Cref{lemma:completegraph}, the induced subgraph $G[V_i]$ is isomorphic to $K_{h_i}$. Therefore, we have that $\mathrm{Aut}\,G[V_i] \cong S_{h_i}$. Let $\phi:H\to S_{h_1}\times \cdots \times S_{h_n}$ be the map defined by sending an automorphism $h\in H$ that acts on each $V_i\in \Pi$ by $\sigma_i$ to the product $\sigma_1\times \sigma_2\times\cdots\times\sigma_n$ for $\sigma_i\in S_{h_i}$. 

It is easy to see that $\phi$ is an isomorphism and so $H\cong S_{h_1}\times \cdots \times S_{h_n}$. Since $G[V_i]\cong K_{h_i}$ for $i\in\{1, 2, \dots,n\}$ and $\mathrm{Aut}\,K_{h_i} \cong S_{h_i}$, we have that  $\mathrm{Aut}\,G[V_1]\times \cdots \times \mathrm{Aut}\,G[V_n]\cong H$. 
\end{proof}

Since we found a normal subgroup, it is a natural step to consider the quotient $(\mathrm{Aut}\,G)/H$ in an effort to establish a connection between $\mathrm{Aut}\,G$ and $\mathrm{Aut}\,\Gamma$.

For the next result, we use Proposition 7.2 from \cite{florez}, which says that taking the base graph $\Gamma$ of a $\Gamma$-reseminant graph $G$ is a well-defined operation. Moreover, by Lemma 7.3 of \cite{florez}, the image of $\Gamma$ under an automorphism $\varphi$ is an isomorphic copy of $\Gamma$. There may be many isomorphic copies of $\Gamma$ in $G$. So, an automorphism maps a copy of $\Gamma$ to another copy, which may be on different vertices.   

\begin{theorem}\label{thm:autsubgroups}
	Given a base graph $\Gamma$ on $n$ vertices, let $G$ be a $\Gamma$-reseminant graph with a partition $\Pi$ containing sets of true twins $V_i$ of size $h_i$ for $i\in\{1,2,\dots,n\}$. Then, $\mathrm{Aut}(G)/H \cong K \subseteq \mathrm{Aut}\Gamma$ for a subgroup $K$ where $H\cong S_{h_1}\times\cdots\times S_{h_n}$. 
\end{theorem}

\begin{proof}
We choose an arbitrary base graph $\Gamma$ of $G$ and label the vertices $1$ through $n$ such that $i\in V_i$ for $i\in\{1, 2, \dots, n\}$. Without loss of generality, we restrict our attention to the induced base graph $G[\{1,2,\dots, n\}]\cong \Gamma$ and construct an automorphism such that $\varphi\in\mathrm{Aut}\,G$ maps $G[\{1,2,\dots, n\}]$ to itself. Let $\hat{\varphi}$ be an automorphism of $G$ such that $\hat{\varphi}(V_i) = V_j$ for some $i,j\in\{1,2,\dots, n\}$ and where $\hat{\varphi}(\ell) = j$ for some vertex $\ell\in V_i$. Let $h_i$ be an automorphism in $H = \{h\in\mathrm{Aut}\,G\,|\,h(V_t) = V_t,\, \forall t\in[1,2,\dots, n]\}$ that acts by the identity permutation on all $V_k\in \Pi$ for $k\neq i$ and acts by the transposition $\tau = (i\,\,\ell)$ on $V_i$. Recall that $H$ is the kernel of the homomorphism in \Cref{thm:normalsubgroup}.

Then, we have that $h_i(i) = \ell$ and so $\hat{\varphi}\circ h_i(i) = j$. Using the same process, we find automorphisms $h_1, \dots, h_n$ such that $\varphi(i) = j$ whenever $\varphi(V_i) = V_j$ for all $i,j\in \{1,2,\dots, n\}$. We set $\varphi = \hat{\varphi}h_1\cdots h_n$. The map $\varphi$ is the composition of automorphisms of $G$, so $\varphi$ is an automorphism. It follows that $\varphi$ is contained in the coset $\hat{\varphi}H\in \mathrm{Aut}(G)/H$ as well. Moreover, the restriction of $\varphi$ to the induced subgraph $G[\{1,2,\dots, n\}]\cong \Gamma$ is an automorphism of $G[\{1,2,\dots, n\}]$ as it maps the induced subgraph to itself. Our choice of automorphism was arbitrary and so for any coset $\varphi H$, there is an automorphism contained in $\varphi H$ such that it maps $G[\{1,2,\dots, n\}]$ to itself. As $G[\{1,2,\dots, n\}]\cong \Gamma$, we refer to $G[\{1,2,\dots, n\}]$ as $\Gamma$ for clarity.

We then construct $\psi:\mathrm{Aut}(G)/H\to \mathrm{Aut}\,\Gamma$ by mapping $\varphi H$ to $\sigma = \phi|_\Gamma$ for an automorphism $\phi\in \varphi H$ such that $\phi$ fixes $\Gamma$. We will show $\psi$ is well-defined. Suppose $\phi_1,\phi_2 \in \varphi H$ such that $\phi_1$ and $\phi_2$ both fix $\Gamma$. Recall that $\mathrm{Aut}(G)/H$ acts faithfully on the partition of true twins $\Pi$ by \Cref{thm:normalsubgroup}. So, the automorphisms contained in the coset $\varphi H$ act identically on $\Pi$; that is, we have that $\phi_1(V_i) = V_j = \phi_2(V_i)$ for all $i,j\in\{1,2, \dots n\}$. By construction, we have also $\phi_1(i) = j = \phi_2(i)$ for all $i,j\in\{1, 2,\dots,n\}$ and so $\phi_1|_\Gamma = \phi_2|_\Gamma$. 

Suppose $\psi(\varphi_1H) = \psi(\varphi_2H) = \sigma$. As $\sigma$ is the restriction $\phi|_\Gamma$ for some automorphism $\phi \in \mathrm{Aut}\,G$, then $\phi$ is in the cosets $\varphi_1H$ and $\varphi_2 H$. But, the cosets of $H$ partition $\mathrm{Aut}\,G$ and so we must have that $\varphi_1 H = \varphi_2 H$. Therefore, the map $\psi$ is injective. 

We show now that $\psi$ is a homomorphism. Suppose $\psi(\varphi_1H) = \sigma_1 = \phi_1|_\Gamma$ and $\psi(\varphi_2H) = \sigma_2 = \phi_2|_\Gamma$. Note that $(\phi_1\phi_2)|_\Gamma = \phi_1|_\Gamma\phi_2|_\Gamma$. We have then $
\psi(\varphi_1H\varphi_2H) = \psi(\varphi_1\varphi_2 H) = (\phi_1\phi_2)|_\Gamma = \phi_1|_\Gamma\phi_2|_\Gamma = \psi(\varphi_1H)\psi(\varphi_2 H).
$

Therefore, $\psi$ is a homomorphism. As $\psi$ is injective and a homomorphism, we have that $\mathrm{Aut}(G)/H \cong K \subseteq \mathrm{Aut}\Gamma$ for $\mathrm{im}(\psi) \cong K \subseteq \mathrm{Aut}\,\Gamma$.  
\end{proof}

\begin{theorem}\label{thm:autforms}
	Let $G$ be a $C_5$-reseminant graph with sets of true twins of sizes $h_1,h_2, \dots, h_5$ for $V_1,V_2,\dots, V_5$, respectively. Then,
	\begin{enumerate}
	\item If $G$ is $k$-regular, then $h_1=h_2=h_3=h_4=h_5=:h$ and $\mathrm{Aut}(G)/H \cong \mathrm{Aut}\,C_5 \cong D_5$ where $H \cong S_h\times S_h\times S_h\times S_h\times S_h$. 
	\item If $G$ is non-regular and there are no reflections of $G$, then $\mathrm{Aut}\,G = H \cong S_{h_1}\times S_{h_2}\times S_{h_3} \times S_{h_4}\times S_{h_5}$. 	
	\item If $G$ is non-regular and has at least one reflection, then $\mathrm{Aut}(G)/H \cong \Z_2$ where $H \cong S_{h_1}\times S_{h_2}\times S_{h_3} \times S_{h_4}\times S_{h_5}$. 	
	\end{enumerate}
\end{theorem}

\begin{proof}
\emph{i)} By \Cref{thm:autsubgroups}, we know $\mathrm{Aut}(G)/H$ is isomorphic to a subgroup of $\mathrm{Aut}\,C_5 = D_5$ where $H \cong S_{h_1}\times S_{h_2}\times S_{h_3} \times S_{h_4}\times S_{h_5}$. Suppose $G$ is $k$-regular where $k=2+3\ell$ for some positive integer $\ell$. By \Cref{thm:resnonregular}, we know all regular graphs generated from $C_5$ are of this form. We can also deduce that all sets of true twins have the same size. By \Cref{thm:autsubgroups}, we know $\mathrm{Aut}(G)/H \cong K\subseteq D_5$. The assertion now follows from the symmetry of the graph.

\medskip

\emph{ii)} We cannot permute the graph by rotations or reflections. However, we can still internally permute $V_i$ for $i\in\{1,2, \dots, 5\}$. It follows then that $\mathrm{Aut}(G) \cong H$. 

\emph{iii)} As $G$ is non-regular, then not all sets $V_i$ are of the same size. If $5\,|\,|\mathrm{Aut}(G)/H|$, then there is an element of order $5$ in $\mathrm{Aut}(G)/H$. Since $\mathrm{Aut}(G)/H \cong K\subseteq D_5$, the element of order $5$ in $\mathrm{Aut}(G)/H$ must correspond to a rotation. But, this implies all $V_i$ can be mapped to each other, which contradicts that they do not have the same size. So, we cannot have that $5|\,|\mathrm{Aut}(G)/H|$. Since there is an axis of symmetric, we do have an element of order 2 corresponding to a reflection. Hence, $2\,|\,|\mathrm{Aut}(G)/H|$ and so $\mathrm{Aut}(G)/H \cong \Z_2$. 
\end{proof}
Given a base graph $\Gamma$, the logic of \Cref{thm:autforms} can be used to study the automorphism groups of any $\Gamma$-reseminant graphs. This is due to \Cref{thm:autsubgroups} since it tells us that automorphism groups of minimal prime graphs generated from base graphs via vertex duplication have strict restrictions on their structure. Mainly, that their automorphism groups always have a direct product of symmetric groups $H$ related to the action of duplication such that quotienting by $H$ gives a group isomorphic to a subgroup of the automorphism group of the starting base graph.


\section{Family of Base Graphs}
Continuing our investigations into $\Gamma$-reseminant graphs, we wanted to find more complex examples of base graphs. We often use $C_5$ as a prototypical example of a base graph and a minimal prime graph due to its size and regular structure. However, there are many graphs which share these properties. We now prove results, which generalize a family of graphs that contain $C_5$ This produces an infinite class of base graphs that are also minimal prime graphs. 
\begin{definition}
    \label{def:TriangleFreeRegularGraph}
    Let $n, k\in \mathbb{N}$. Define $G_{n,k}$ to be the graph with vertex set 
    \begin{equation*}
        V=\lbrace 0, 1, \dots, n-1\rbrace
    \end{equation*}
    and there is an edge between the vertices $i,j$ if and only if
    at least one of $i-j$ or $j-i$ is in the set $\{\pm k, \pm(k+1), \dots, \pm(2k-1)\}$.
\end{definition}
\begin{lemma}
    \label{lem:TFRG}
    Suppose $n,k\in \mathbb{N}$ are such that $n\ge 6k-2$ and $n\neq 3,7,9$. Then $G_{n,k}$ is a triangle-free regular graph whose degree is equal to $2k$.
\end{lemma}
\begin{proof}
    This follows immediately from the results in \cite{sidorenko}.
\end{proof}
For the purposes of minimal prime graphs, we now how to verify properties about the $3$-colorability of these graphs.
\begin{lemma}
    \label{lem:TFRG3Colorable}
    Let $n \ge 5$ with $n\equiv 0, 5\mod 6$ and $k=\lfloor (n+2)/6\rfloor$. Then $G_{n,k}$ is 3-colorable.
\end{lemma}

We believe it is the case that all such graphs have chromatic number 3, but we only require the weaker $3$-colorability.
\begin{proof}
    We first show the statement holds for $n\equiv 0\mod 6$. We note in this case that $n=6k$, and so our vertices are labeled from $0$ to $6k-1$.
   
    Letting the set of colors be $\lbrace 0, 1, 2\rbrace$, we take the following function as our coloring:
    \begin{equation*}
        f(v)=
        \begin{cases}
            0 & 0\le v < k\\
            1 & k\le v < 2k\\
            2 & 2k\le v < 3k\\
            0 & 3k\le v < 4k\\
            1 & 4k\le v < 5k\\
            2 & 5k\le v < 6k
        \end{cases}
    \end{equation*}
    We see by symmetry it suffices to verify no vertices $v$ with $0\le v<k$ are adjacent to a vertex colored $0$. We see for a vertex $u$ also with $0\le u<k$ we have $(v-u)\in \{0, \pm1, \dots, \pm(k-1)\}$, and so none are adjacent. We see for a vertex $u$ instead with $3k\le u<4k$ that we have $(v-u)\in \{\pm(2k+1), \dots, \pm(4k-1)\}$, and so again no pair is adjacent. 

    We now show the statement holds for $n\equiv 5\mod 6$. We note in this case that $n=6k-1$, and so our vertices are labeled from $0$ to $6k-2$.
   
    Letting the set of colors be $\lbrace 0, 1, 2\rbrace$, we take the following function as our coloring:
    \begin{equation*}
        f(v)=
        \begin{cases}
            0 & 0\le v < k\\
            1 & k\le v < 2k\\
            2 & 2k\le v < 3k\\
            0 & 3k\le v < 4k\\
            1 & 4k\le v < 5k\\
            2 & 5k\le v < 6k-1
        \end{cases}
    \end{equation*}
    We see no verties colored $0$ or $1$ can be adjacent for the same reason as in the $n\equiv 0\mod 6$ case. Consider any two vertices colored $2$. We see for a similar reason as to the other colors that if they are both within the same range of $[2k, 3k)$ or $[5k, 6k-1)$ they are not adjacent. In the event the two vertices are in different ranges, we see points in the $2k$ to $3k-1$ range and points in the $5k$ to $6k-2$ range are at least a distance of $6k-2-(3k-1)=3k-1$ apart. As $3k-1>2k-1$ we conclude no edges exist, and therefore no edge can exist.  
\end{proof}
\begin{corollary}
    \label{cor:TFRGPrimeGraph}
    Let $n\ge 5$ with $n\equiv 0, 5\mod 6$ and $k=\lfloor (n+2)/6\rfloor$. Then $\overline{G_{n,k}}$ is a solvable prime graph.
\end{corollary}
\begin{proof}
    Follows directly from $G_{n,k}$ being triangle-free by \Cref{lem:TFRG}, noting that our choices of $n$ cause us to avoid the edge cases in said lemma, and 3-colorable by \Cref{lem:TFRG3Colorable}.
\end{proof}

\begin{theorem}
    \label{thm:}
    Let $n\ge 5$ with $n\equiv 0, 5\mod 6$ and $k=\lfloor (n+2)/6\rfloor$. Then $\overline{G_{n,k}}$ is a minimal prime graph.
\end{theorem}
\begin{proof}
    Since $G_{n,k}$ is circulant graph, it suffices to show that adding any edge $(0, m)$ either introduces a triangle or increases the chromatic number for $1\le m< k$ or $2k\le m\le 3k$. We will show it introduces a triangle.

    For $1\le m< k$, we see $m$ is adjacent to the vertices $k+m, \dots, 2k-1+m$. We see that this range overlaps $k, \dots, 2k-1$ for all $m$, and so there is some vertex adjacent to both $0$ and $m$, and thus there is a triangle.

    For $2k\le m\le 3k$, we see $m$ is adjacent to the vertices $m-k, \dots, m-2k+1$. We see that this range overlaps $k, \dots, 2k-1$ for all $m$, and so there is some vertex adjacent to both $0$ and $m$, and thus there is a triangle.
\end{proof}

\begin{theorem}
    \label{thm:bsegraph}
    Let $n\ge 5$ with $n\equiv 0, 5\mod 6$ and $k=\lfloor (n+2)/6\rfloor$. Then $\overline{G_{n,k}}$ is a base graph
\end{theorem}
\begin{proof}
    By the construction of $\overline{G_{n,k}}$, we see no vertices share the same set of neighboring vertices, and thus no vertices can be true twins, indicating $\overline{G_{n,k}}$ cannot be a graph generated through vertex duplication.
\end{proof}


\section{Products of Solvable Prime Graphs}

As part of our investigations into the above properties of minimal prime graphs, we wished to find examples of larger minimal prime graphs which would have more complex structure for us to investigate. Brute-force checking for such graphs is slow, due to the computational complexity of determining $3$-colorability. Another method of creating graphs with large vertex counts are various graph products. Investigating graph products to create larger minimal prime graph examples produced the following results on the relation of these various graph products to minimal prime graphs, and the more general solvable prime graphs.




In this section, we study the effect the direct product, Cartesian product, and strong product have on solvable prime graphs. Let $G = (V(G), E(G))$ and $H = (V(H), E(H))$ be arbitrary graphs. We will use the below definitions.
\begin{itemize}
    \item The \emph{direct (or tensor or Kronecker) product}  $G\times H$ of graphs $G$ and $H$ is defined as the graph with vertex set $V(G)\times V(H)$ and $\{(u,v),(u',v')\}$ is an edge in $G\times H$ if and only if $\{u,u'\}\in E(G)$ and $\{v,v'\}\in E(H)$. 
    \item The \emph{Cartesian product} $G\mathop{\Box}H$ is defined as the graph with vertex set $V(G)\times V(H)$ and $\{(u,v),(u',v')\}$ is an edge in $G\times H$ if and only if either $u=u'$ and $\{v,v'\}\in E(H)$ or $v=v'$ and $\{u,u'\}\in E(H)$.
    \item The \emph{strong product} $G\boxtimes H$ is defined as the graph with vertex set $V(G)\times V(H)$ and edge set $E(G\boxtimes H) = E(G\times H) \cup E(G\mathop{\Box}H)$. 
    \item The \emph{adjacency matrix} $A_G$, or $A(G)$, of a graph $G$ is the 0-1 matrix indexed by $V(G)$, where $A_{uv} = 1$ when there is an edge $\{u,v\}\in E(G)$ and 0 otherwise. 
    \item We define the complementary direct product $\thickbar \times$ as $G\thickbar{\times} H = \overline{\overline{G} \times \overline{H}}$.
    \item We define $G \, \thickoverline \Box \, H = \overline{\overline{G}\mathop{\Box} \overline{H}}$.
\end{itemize}
Some common facts that will also be useful to us are: 
\begin{itemize}
    \item The adjacency matrix of the direct graph $G\times H$ is the matrix $A_G\otimes A_H$, i.e. the tensor product of their adjacency matrices.
    \item The number of triangles in $G$ is given by $\mathrm{tr}(A_G^3)/6$.
    \item For the chromatic number $\chi$, the following inequalities hold for graphs $G$ and $H$: $\chi(G\times H) \leq \min(\chi(G), \chi(H)), \chi(G\mathop{\Box} H) = \max\{\chi(G), \chi(H)\},$ and $\chi(G\boxtimes H) \leq \chi(G)\chi(H)$. 
\end{itemize}

\begin{theorem}\label{thm:primegraphinduct}
	If $\Gamma_1 , \dots, \Gamma_n$ are solvable prime graphs, then $\displaystyle \thickoverline \bigtimes_{i=1}^n \Gamma_i$ is a solvable prime graph.
\end{theorem}

\begin{proof}
	Let $n=2$. We denote $A_{\Gamma_1}$ as $A$ and $A_{\Gamma_2}$ as $B$. Since $\Gamma_1$ and $\Gamma_2$ are triangle-free, we have that $\tr(A^3)=\tr(B^3)=0$. Using a tensor product identity, we have that $(A\otimes B)^3 = A^3\otimes B^3.$ Using this and the property that $\tr(A\otimes B) = \tr(A)\cdot \tr(B)$, we have that: $\tr((A\otimes B)^3) = \tr(A^3)\tr(B^3) = 0$. Therefore, the graph $\Gamma_1\times \Gamma_2$ is triangle-free.  
 
    Now, consider the chromatic number $\chi$ of $\Gamma_1\times \Gamma_2$. We have $$\chi(\Gamma_1\times \Gamma_2) \leq \min(\chi(\Gamma_1), \chi(\Gamma_2))\leq3 $$
    Therefore, the direct product $\Gamma_1\times \Gamma_2$ is 3-colorable and so is a solvable prime graph. 
	
	Assume $k\geq 2$ and $\bigtimes_{i=1}^k \overline{\Gamma_i}$ is triangle-free and 3-colorable. We then have:$$
	\left(\bigtimes_{i=1}^k \overline{\Gamma}_i\right)\times \overline{\Gamma}_{k+1} = \bigtimes_{i=1}^{k+1} \overline{\Gamma}_i$$
	The left-hand side is triangle-free and 3-colorable, so $\bigtimes_{i=1}^{k+1} \overline{\Gamma}_i$ is as well. By induction, this follows for all $n\in\N$. 
\end{proof}

Although taking the direct product of graphs preserves $3$-colorability and triangle-free properties of graphs, it does not preserve minimality. We prove this by looking at how $C_5 \times C_5$ shows up in direct products of minimal prime graphs and how this interacts with the rest of the graph.

\begin{remark}\label{remark:C_5timesC_5}
$C_{5} \times C_5$ is not the complement of a minimal prime graph. If the vertices of $C_5$ are labeled $\{v_1, v_2, v_3, v_4, v_5\}$ and the vertices of $C_5 \times C_5$  are labeled $\{(v_i, v_j) | v_i, v_j \in \{v_1, \dots, v_5\} \}$, then there are edges that can be added without creating a triangle or causing a $4$-coloring. To see this, the edge $\{(v_1,v_2),(v_2,v_1)\}$ is already an edge in the graph, so the vertices $(v_1,v_2)$ and $(v_2, v_1)$ must always be in different color partitions, so one of the two must be in a different color partition from $(v_2, v_2)$. Therefore, adding one of the edges $\{(v_1, v_2),(v_2,v_2)\}$ or $\{(v_2, v_1), (v_2,v_2)\}$ will not cause a new coloring (whichever edge is necessarily already a different color). Also, $(v_2, v_2)$ shares no adjacent edges with $(v_1, v_2)$ or $(v_2, v_1)$, so the resulting graph is also triangle-free. 
\end{remark}

To begin proving that the direct product of solvable prime graphs will not be a minimal prime graph, some lemmas regarding general direct products of graphs will help.

\begin{lemma}\label{lemma:gammasubgraphoftensor}
$\Gamma \times \Gamma$ contains a subgraph isomorphic to $\Gamma$
\end{lemma}

\begin{proof}
Define $\Gamma^*$ as the induced subgraph on $\Gamma \times \Gamma$ with the vertex set $V(\Gamma^*) = \{(u,u) \mid u \in V(\Gamma)\}$, and then define a function $\Phi : \Gamma^* \rightarrow \Gamma$ where $\Phi((u,u)) \rightarrow u$. It is easy to confirm that $\Phi$ is a well-defined graph isomorphism which shows that $\Gamma \times \Gamma$ contains induced subgraph $\Gamma^*$ isomorphic to $\Gamma$.
\end{proof}

\begin{corollary}
For $n \geq 2$, the graph $\bigtimes_{i=1}^n \Gamma$ contains an induced subgraph isomorphic to $\Gamma \times \Gamma$ 
\end{corollary}

\begin{proof}
The base case, $\bigtimes_{i=1}^2 = \Gamma \times \Gamma$ is trivial as the identity forms the isomorphism.

For $n \geq 3$, assuming $\bigtimes_{i=1}^n \Gamma$ contains an induced subgraph isomorphic to $\Gamma \times \Gamma$ and using the previous lemma, this induced subgraph will contain an induced subgraph $\gamma$ isomorphic to $\Gamma$, related by isomorphism $\phi : \gamma \rightarrow \Gamma$. This time, define $\Gamma^*$ as the induced subgraph on $\bigtimes_{i=1}^{n+1} \Gamma = \bigtimes_{i=1}^{n} \Gamma \times \Gamma$ with the vertex set $V(\Gamma^*) = \{(u, v) \mid u \in V(\gamma), v \in V(\Gamma)\}$, and define a function $\Phi : \Gamma^* \rightarrow \Gamma \times \Gamma$ where $\Phi((u,v)) \rightarrow (\phi(u),v)$.

It is again easy for the reader to confirm that $\Phi$ is a well-defined and invertible. To show it is a graph homomorphism, use the fact that $\phi$ is a graph isomorphism which shows that $\{u_1, u_2\} \in E(\gamma) \leftrightarrow \{\phi(u_1), \phi(u_2)\} \in E(\Gamma)$ and therefore $\{(u_1, v_1), (u_2, v_2)\} \in E(\Gamma^*) \leftrightarrow \{\phi(u_1), \phi(u_2)\}, \{v_1, v_2\} \in E(\Gamma)$, so by the definition of a graph direct product, $\{\Phi(u_1,v_1), \Phi(u_2, v_2)\} \in E(\Gamma \times \Gamma) \leftrightarrow \{(u_1, v_1), (u_2, v_2)\} \in E(\Gamma^*)$ which finishes the proof that $\Phi$ is a graph isomorphism and $\Gamma^*$ is isomorphic to $\Gamma \times \Gamma$.
\end{proof}

The next step in negating minimality in direct products of minimal prime graphs is to combine the fact, proved in \cite{gruber}, that all minimal prime graphs contain an induced subgraph isomorphic to $C_5$ with the above results.

\begin{lemma}\label{tensorisnotMPG}
If $n \geq 2$ and $\Gamma_i$ are all minimal prime graphs, then $\displaystyle \thickoverline \bigtimes_{i=1}^n \Gamma_i$ is never a minimal prime graph.
\end{lemma}

\begin{proof}
By Lemma 4.1 of \cite{gruber}, every minimal prime graph $\Gamma$ contains an induced subgraph $\Gamma^*$ isomorphic to $C_5$. As $C_5$ is self-complementary, the complement $\overline{\Gamma^*}$, an induced subgraph of $\overline{\Gamma}$, will be isomorphic to $C_5$. So for each $\overline{\Gamma_i}$, there exists an induced subgraph, $\overline{\Gamma_i^*}$ with graph isomorphism $\phi_i:\overline{\Gamma_i^*} \rightarrow C_5$. For $n=2$, $\overline{\Gamma_1} \times \overline{\Gamma_2}$ will have induced subgraph $\overline{\Gamma^*}$ defined by the vertices $\{(u, v) \mid u \in \overline{\Gamma_1^*}, v \in \overline{\Gamma_2^*} \}$. We define $\Phi: \overline{\Gamma^*} \rightarrow C_5 \times C_5$ by $\Phi((u,v)) = (\phi_1(u), \phi_2(v))$. It can easily be verified that $\Phi$ is a well-defined graph isomorphism, showing $C_5 \times C_5$ is isomorphic to an induced subgraph of $\overline{\Gamma_1} \times \overline{\Gamma_2}$. 

For $n \geq 2$, assume $\bigtimes_{i=1}^n \overline{\Gamma_i}$ contains an induced subgraph isomorphic to $C_5 \times C_5$. By \Cref{lemma:gammasubgraphoftensor}, it will contain an induced subgraph, $\overline{\gamma}\cong C_5$ and related by a graph isomorphism $\phi_{\gamma}$. Also, any minimal prime graph $\overline{\Gamma_{n+1}}$ will contain an induced subgraph $\overline{\gamma_{n+1}} \cong C_5$ by the graph isomorphism $\phi_{n+1}$. Therefore $\bigtimes_{i=1}^{n+1} \overline{\Gamma_i} = (\bigtimes_{i=1}^n \overline{\Gamma_i}) \times \overline{\Gamma_{n+1}}$ will have induced subgraph $\overline{\Gamma^*}$ defined by the vertex set $V(\overline{\Gamma^*}) = \{(u, v) \mid u \in \overline{\gamma}, v \in \overline{\gamma_{n+1}}\}$. Define a function $\Phi : \overline{\Gamma^*} \rightarrow C_5 \times C_5$ by $\Phi((u,v)) = (\phi_{\gamma}(u), \phi_{n+1}(v))$. The proof that $\Phi$ is a well defined graph isomorphism follows all the same steps as the prior cases, and therefore $\overline{\Gamma^*}$ is isomorphic to $C_5 \times C_5$.

This shows that for all $n \geq 2$, $\bigtimes_{i=1}^n \overline{\Gamma_i}$ contains an induced subgraph isomorphic to $C_5 \times C_5$. As explained in \Cref{remark:C_5timesC_5}, there are numerous edges which can be added that do not create a triangle or violate three colorability, so $\bigtimes_{i=1}^n \overline{\Gamma_i}$ cannot be the complement of a minimal prime graph.
\end{proof}

From this result, we learn that not only is it impossible for the complement direct product of minimal prime graphs to be minimal, but it is impossible for the complement direct product of any solvable prime graphs to be minimal. This immediately provides a proof for \Cref{thm:resnonregular}.

\begin{theorem}
If $n \geq 2$ and $\Gamma_i$ are all solvable prime graphs, then $\displaystyle \thickoverline \bigtimes_{i=1}^n \Gamma_i$ is never a minimal prime graph.
\end{theorem}

\begin{proof}

This proof is done by case work looking at the number of vertices and connectivity.

Starting in the case where all $\Gamma_i$ are connected and have five or more vertices, each $\Gamma_i$ will have a subgraph $\gamma_i$ which is a minimal prime graph defined on the same vertices so that $V(\gamma_i) = V(\Gamma_i)$ and $E(\gamma_i) \subseteq E(\Gamma_i)$. As all edges not in $\Gamma_i$ will not be in $\gamma_i$, it is clear $V(\overline{\gamma_i}) = V(\overline{\Gamma_i})$ and $E(\overline{\gamma_i}) \supseteq E(\overline{\Gamma_i}).$ This implies that if $u = (u_1, \dots, u_n) \in E(\bigtimes_{i=1}^{n} \overline{\Gamma_i})$, defined by the direct product so that $u_i \in E(\overline{\Gamma_i})$, then all $ u_i \in E(\overline{\gamma_i})$ and $ u \in E(\bigtimes_{i=1}^n\overline{\gamma_i}) $. Shown in \Cref{tensorisnotMPG}  $\bigtimes_{i=1}^n \overline{\gamma_i}$ is not a minimal prime graph, so there exists an edge $v \notin E(\bigtimes_{i=1}^n \overline{\gamma_i})$ which would not create a triangle or violate the three coloring if added to $\bigtimes_{i=1}^n\overline{\gamma_i}$. As $\bigtimes_{i=1}^n \overline{\Gamma_i}$ is a subgraph, adding the edge $v$ could not create a triangle or violate the three coloring in the subgraph. This shows $v$ is an edge that contradicts minimality, and therefore $\bigtimes_{i=1}^n \overline{\Gamma_i}$ is not the complement to a minimal prime graph.

If any $\Gamma_i$ are not connected, then $\overline{\Gamma_i}$ must be bipartite. This is because every vertex of a component must have an edge to every vertex of the other in the complement, and this also means that there can be no edges between vertices of the same component without creating a triangle. Clearly, there could be at most two components in a solvable prime graph, so each component is a color making $\overline{\Gamma_i}$ bipartite. In Lemma 2 of \cite{bipartitegraphtensor}, it is shown that a direct product of two graphs is bipartite if and only if at least one of the graphs is bipartite. Therefore $\bigtimes_{i=1}^n \overline{\Gamma_i}$ will be bipartite if any $\Gamma_i$ are not connected. $C_5$ is not bipartite, so $C_5$ could not embed in a bipartite graph. As all complements of minimal prime graphs contain an induced subgraph of $C_5$, $\bigtimes_{i=1}^n\overline{\Gamma_i}$ could not be the complement of a minimal prime graph.

All cases where any $\Gamma_i$ has less than five vertices, will cause a bipartite $\overline{\Gamma_i}$, so $\bigtimes_{i=1}^n \overline{\Gamma_i}$ will be bipartite and cannot be the complement of a minimal prime graph.

Any $\overline{\Gamma_i}$ with one or two vertices, is two-colorable because there are not enough vertices to force a higher coloring. Any $\overline{\Gamma_i}$ with three or four vertices must also be bipartite to avoid a triangle. 

This completes all cases and shows that no direct product of solvable prime graphs will be minimal.
\end{proof}

We show next that a similar theorem holds for the Cartesian product. Note that if $G$ and $H$ have $n_1$ and $n_2$ vertices, respectively, then the adjacency matrix of $G\mathop{\Box} H$ is $$ A_{G\Box H}	= (A_G\otimes I_{n_2}) + (I_{n_1}\otimes A_H)$$ for identity matrices $I_{n_1}$ and $I_{n_2}$.

\begin{theorem}
	If $\Gamma_1, \dots, \Gamma_n$ are solvable prime graphs, then $ \thickoverline \msquare_{i=1}^n \Gamma_i$ is a solvable prime graph. 
\end{theorem}

\begin{proof}
	We proceed by induction. As $n=1$ is immediate, consider the base case of $n=2$. Let $\Gamma_1$ and $\Gamma_2$ be solvable prime graphs on $n_1$ and $n_2$ vertices, respectively. Let $A_1$, $A_2$, and $A_{1\Box 2}$ be the adjacency matrices of $\overline{\Gamma_1}$, $\overline{\Gamma_2}$, and $\overline{\Gamma_1} \Box \overline{\Gamma_2}$,  respectively. As $\overline{\Gamma_1}$ and $\overline{\Gamma_2}$ are triangle-free, we know $\tr(A_1^3) = \tr(A_2^3)=0$. 
\[
\begin{split}
	\tr(A_{1\Box2}^3) & = \tr\big[(A_1\otimes I_{n_2} + I_{n_1}\otimes A_2)^3\big] \\
	& = \tr\big[A_1^3\otimes I_{n_2} +  3(A_1^2\otimes A_2) + 3(A_1\otimes A_2^2) + I_{n_1}\otimes A_2^3\big] \\
	& = \tr(A_1^3)\tr(I_{n_2})+ 3\tr(A_1^2)\tr(A_2) + 3\tr(A_1)\tr(A_2^2) + \tr(I_{n_1})\tr(A_2^3) \\
	& = 0 \\
\end{split}
\]
The trace of $A_1$ and $A_2$ are 0 because all of our graphs are assumed to not have loops. Therefore, the Cartesian product $\overline{\Gamma_1} \Box \overline{\Gamma_2}$ is triangle-free. We can conclude that $\overline{\Gamma_1} \Box \overline{\Gamma_2}$ is 3-colorable by $\chi(\overline{\Gamma_1}\mathop{\Box}\overline{\Gamma_2}) = \max\{\chi(\overline{\Gamma_1}), \chi(\overline{\Gamma_2})\}\leq 3.$ It follows then that $\overline{\overline{\Gamma_1}\Box \overline{\Gamma_2}}$ is a solvable prime graph.

Assume that $\overline{\mathop{\msquare}_{i=1}^k \overline{\Gamma_i}}$ is a solvable prime graph for $k\geq 2$. Let $G = \overline{\mathop{\msquare}_{i=1}^k \overline{\Gamma_i}}$ and and so $G$ is triangle-free and 3-colorable. Consider an arbitrary prime graph $\Gamma_{k+1}$ of a finite solvable group. By the base case, we see that $\overline{G}\Box \overline{\Gamma}_{k+1}$ is triangle-free and 3-colorable. The result follows immediately. By induction, we have that this holds for all $n\in \N$. 
	\end{proof}

The strong product does not have an analogous result to the direct and Cartesian products. For example, consider $C_5\boxtimes C_5$. This gives $\chi(C_5\boxtimes C_5) = 5$ and $\chi(\overline{C_5\boxtimes C_5}) = 8$. Both the product and the complement of the product are also not triangle-free. However, we can proceed in a slightly different direction to find an interesting result. 

Since $\chi(G\boxtimes H) \leq \chi(G)\chi(H)$, if $\chi(H)=1$ and $\chi(G)\leq 3$, then $\chi(G\boxtimes H)\leq 3$. Before we prove a result related to the strong product, we illustrate the idea behind the proof. Consider the graphs below. 

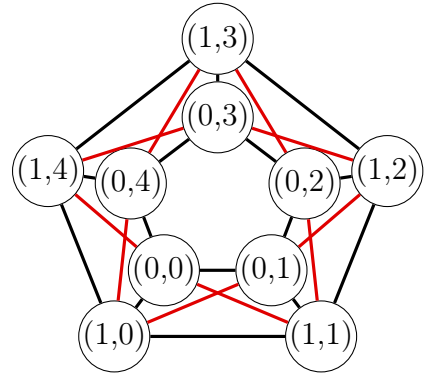
\begin{figure*}[ht!]
\begin{minipage}{.48\textwidth}
\centering
\caption{$C_5$-reseminant graph $G$.}
\label{reggraph}
\begin{tikzpicture}[scale=0.55,
dot/.style = {circle,draw, inner sep=0, minimum size=1.5em},
dot/.default = 6pt  
                    ] 
\begin{scope}[every node/.style={dot}]
     \node (5) at (-.7,.7) {4};
    \node (10) at (-2.7,1) {9};
    
    \node (1) at (0.1,-1.4) {0};
    \node (6) at (-1.1,-3) {5};
    
    \node (2) at (2.7,-1.4) {1};
 	\node (7) at (3.9,-3) {6};

	\node (3) at (3.5,.7) {2};
	\node (8) at (5.5,1) {7};
	
    \node (4) at (1.4,2.3) {3};
    \node (9) at (1.4,4.2) {8};  \end{scope}

\begin{scope}[
              every edge/.style={draw=black,very thick}]
    \path [-] (1) edge node {} (2);
     \path [-] (2) edge node {} (3);
     \path [-] (3) edge node {} (4);
     \path [-] (4) edge node {} (5);
     \path [-] (5) edge node {} (1);
     \path [-] (6) edge node {} (1);
     \path [-] (6) edge node {} (2);
     \path [-] (6) edge node {} (5);
     \path [-] (7) edge node {} (2);
     \path [-] (7) edge node {} (3);
     \path [-] (7) edge node {} (1);
     \path [-] (8) edge node {} (3);
     \path [-] (8) edge node {} (2);
     \path [-] (8) edge node {} (4);
     \path [-] (9) edge node {} (4);
     \path [-] (9) edge node {} (3);
     \path [-] (9) edge node {} (5);
     \path [-] (10) edge node {} (5);
     \path [-] (10) edge node {} (4);
     \path [-] (10) edge node {} (1);
     \path [-] (6) edge node {} (7);
     \path [-] (8) edge node {} (7);
     \path [-] (9) edge node {} (8);
     \path [-] (10) edge node {} (9);
     \path [-] (10) edge node {} (6);
\end{scope}
\end{tikzpicture}
\end{minipage}
\begin{minipage}{.48\textwidth}
\centering
\caption{$K_2\boxtimes C_5$.}
\begin{tikzpicture}[scale=0.55,
dot/.style = {circle,draw, inner sep=0, minimum size=1.5em},
dot/.default = 6pt  
                    ] 
\begin{scope}[every node/.style={dot}]
     \node (5) at (-.7,.7) {(0,4)};
    \node (10) at (-2.7,1) {(1,4)};
    
    \node (1) at (0.1,-1.4) {(0,0)};
    \node (6) at (-1.1,-3) {(1,0)};
    
    \node (2) at (2.7,-1.4) {(0,1)};
 	\node (7) at (3.9,-3) {(1,1)};

	\node (3) at (3.5,.7) {(0,2)};
	\node (8) at (5.5,1) {(1,2)};
	
    \node (4) at (1.4,2.3) {(0,3)};
    \node (9) at (1.4,4.2) {(1,3)};  \end{scope}

\begin{scope}[
              every edge/.style={draw=black,very thick}]
     \path [-] (1) edge[style={draw=black}] node {} (2);
     \path [-] (2) edge[style={draw=black}] node {} (3);
     \path [-] (3) edge[style={draw=black}] node {} (4);
     \path [-] (4) edge[style={draw=black}] node {} (5);
     \path [-] (5) edge[style={draw=black}] node {} (1);
     \path [-] (6) edge[style={draw=black}] node {} (1);
   
     \path [-] (2) edge[style={draw=black}] node {} (7);
     \path [-] (3) edge[style={draw=black}] node {} (8);
	 \path [-] (4) edge[style={draw=black}] node {} (9);
     \path [-] (5) edge[style={draw=black}] node {} (10);
      
     \path [-] (6) edge[style={draw=black}] node {} (7);
     \path [-] (8) edge[style={draw=black}] node {} (7);
     \path [-] (9) edge[style={draw=black}] node {} (8);
     \path [-] (10) edge[style={draw=black}] node {} (9);
     \path [-] (10) edge[style={draw=black}] node {} (6);

     \path [-] (6) edge[style={draw=nred}] node {} (2);
     \path [-] (6) edge[style={draw=nred}] node {} (5);
     
     \path [-] (7) edge[style={draw=nred}] node {} (3);
     \path [-] (7) edge[style={draw=nred}] node {} (1);
     
     \path [-] (8) edge[style={draw=nred}] node {} (2);
     \path [-] (8) edge[style={draw=nred}] node {} (4);
     
     \path [-] (9) edge[style={draw=nred}] node {} (3);
     \path [-] (9) edge[style={draw=nred}] node {} (5);
     
     \path [-] (10) edge[style={draw=nred}] node {} (4);
     \path [-] (10) edge[style={draw=nred}] node {} (1);

\end{scope}
\end{tikzpicture}
\end{minipage}
\end{figure*}

By observation, we can see that the graph $G$ generated from $C_5$ by duplicating each of the vertices in $C_5$ once is isomorphic to $K_2\boxtimes C_5$. This follows from the fact that $K_2\boxtimes C_5$ has vertex set $V(K_2)\times V(C_5)$ and edge set $E(K_2\times C_5)\cup E(K_2\mathop{\Box} C_5)$. This leads to the following theorem.

\begin{theorem}
Let $G$ be a minimal prime graph. Then, $K_{m+1}\boxtimes G$ is a minimal prime graph. Furthermore, the graph $K_{m+1}\boxtimes G$ is isomorphic to the graph generated from duplicating each vertex of $G$ exactly $m$ times.
\end{theorem}

\begin{proof}
Let $G$ be a minimal prime graph on $n$ vertices. We label $G$ from $0$ to $n-1$. Let $G'$ be the graph generated from duplicating each vertex of $G$ exactly $m$ times. We construct a map $\psi:K_{m+1}\boxtimes G\to G'$ by sending the vertex $(i,j)$ for $i\in\{0,1, \dots, m\}$ and $j\in\{0,1, \dots, n-1\}$ to $v$ where $v$ is the $i$-th vertex duplication of $j$. For $i=0$, we send $(0,j)$ to the vertex $j\in\{0,1, \dots, n-1\}$ in the starting graph $G$. Each vertex $(i,j)\in V(K_{m+1}\boxtimes G)$ has exactly one image, as each vertex in $G$ was duplicated exactly $m$ times. So, this is a well-defined function. 

We show now that $\psi$ is bijective and a graph homomorphism. Suppose $\psi(i_1,j_1) = \psi(i_2, j_2) = v$, then $v$ is both the $i_1$-th duplicated vertex of $j_1$ and the $i_2$-th duplicated vertex of $j_2$. But, as each vertex in $V(G)$ has $m$ unique duplicated vertices, then we must have that $j_1=j_2$. As $v$ is both the $i_1$-th and $i_2$-th duplicated vertex of $j_1$, then we also have $i_1 = i_2$. From this, it follows that $(i_1, j_1) = (i_2, j_2)$. 

Suppose $v\in V(G')$. Then, either $v$ was a vertex in the starting graph $G$ or $v$ is a duplicate of a vertex $j$ in $G$ for $j\in\{0,1, \dots, n-1\}
    $. In the first case, we have that $\psi(0,v) = v$. In the second, $v$ must be the $i$-th duplicate of $j$ for some $i\in \{0,1, \dots, m\}$ and so $\psi(i,j) = v$. We have then that $\psi$ is bijective. 

Next, we show that $\psi$ is a graph homomorphism. Suppose $\{(i_1, j_1), (i_2, j_2)\}$ is an edge in $K_{m+1}\boxtimes G$. We need to show that $\{\psi(i_1, j_1), \psi(i_2, j_2)\} = \{v_1, v_2\}$ is an edge in $G'$. We need to consider this in cases due to the nature of the strong product. 

\emph{Case 1:} $i_1 = i_2$

Since $i_1=i_2$ and $\{(i_1, j_1), (i_2, j_2)\}$ is an edge in $K_{m+1}\boxtimes G$, we know that $\{j_1, j_2\}$ is an edge in $G$. As $\{j_1, j_2\}$ is an edge in $G$, there is an edge between every duplicated vertex of $j_1$ and of $j_2$. The images of $\psi(i_1, j_1)$ and $\psi(i_2, j_2)$ are duplicated vertices of $j_1$ and $j_2$. From this, we know that $\{\psi(i_1, j_1), \psi(i_2, j_2)\} = \{v_1, v_2\}$ is an edge in $G'$.

\medskip

\emph{Case 2:} $\{i_1,i_2\}\in E(K_{m+1})$ and $\{j_1, j_2\}\in E(G)$

As $\{j_1, j_2\}$ is an edge in $G$, every pair of duplicated vertices of $j_1$ and of $j_2$ are connected with an edge. This follows from the fact that duplicated vertices share the same closed neighborhood. But, $\psi(i_1, j_1)=v_1$ and $\psi(i_2, j_2)=v_2$ are duplicated vertices of $j_1$ and $j_2$ and so $\{v_1, v_2\}$ is an edge in $G'$.

\medskip

\emph{Case 3:} $j_1 = j_2$

Since $j_1 = j_2$, the images $\psi(i_1, j_1)=v_1$ and $\psi(i_2, j_2)=v_2$ are duplicated vertices of $j_1$. Every two duplicated vertices of $j_1$ are connected with an edge and so $\{v_1, v_2\}$ is an edge in $G'$.

From the previous 3 cases, we get that $\psi$ is a graph homomorphism. The map $\psi$ is a bijective graph homomorphism and so $G'\cong K_{m+1}\boxtimes G$. Since $G'$ is a minimal prime graph, the graph $K_{m+1}\boxtimes G$ is a minimal prime graph as well. 
\end{proof}

To summarize the results of this section:

\begin{itemize}
    \item If $G$, $H$ are solvable prime graphs, then $G\,\thickoverline \msquare H$ and $G\,\thickoverline \times H$ are solvable prime graphs.
    \item If $G$, $H$ are solvable prime graphs, then $G\,\thickoverline \times H$ is not a minimal prime graph.
    \item If $G$ is a minimal prime graph, then $K_n \boxtimes G$ is a minimal prime graph isomorphic to the graph generated from duplicating each vertex $n-1$ times. 
\end{itemize}


\section{Clique Generation for Minimal Prime Graphs}

In the prior pages, we have focused exclusively on results connected to vertex duplication, building on previous work. However, not all minimal prime graphs can be built through vertex duplication.
Previously in \cite{florez} the idea of \emph{generating} a prime graph was developed, where a minimal prime graph is generated from another by adding a vertex along with edges. Expanding the minimal prime graph generation methods in \cite{florez}, we introduce a novel approach, separate from vertex duplication, that generates new minimal prime graphs from existing ones. 

\begin{definition}[Minimal Prime Graph Generation]
    \label{def:MinimalPrimeGraphGeneration}
    Suppose $\Gamma$ is a minimal prime graph. We say that any minimal prime graph $\Gamma^{\prime}$ is \emph{generated} from $\Gamma$ if there exists a vertex $w\in \Gamma^{\prime}$ such that $\Gamma\cong \Gamma^{\prime}\cut\lbrace w\rbrace$.
\end{definition}
We are interested in the problem of classifying the minimal prime graphs $\Gamma^{\prime}$ which can be generated from a given minimal prime graph $\Gamma$, as well as classifying common structures used in these generated graphs (i.e. vertex duplication). Vertex duplication serves as a simple example of generating minimal prime graphs. However, rephrasing this method of generating a minimal prime graph from another to be relative to the set of vertices the new vertex is adjacent to, rather than relative to a specific vertex, will lend itself better to a generalization of this phenomena. We require this generalization as it is easily seen there exists generated minimal prime graphs which is not done through this process of vertex duplication. See \Cref{fig:cgnotvd} for an example. As such, it means this vertex duplication generation method is not general enough. Shifting our perspective away from the idea of duplicating a vertex, in general we are interested in these subsets of vertices $U$ of a minimal prime graph $\Gamma$ for which the graph
\begin{equation*}
    \Gamma^{\prime}=(V(\Gamma)\cup\lbrace w \rbrace, E(\Gamma)\cup\lbrace \lbrace w, u\rbrace\mid u\in U\rbrace)
\end{equation*}
is a minimal prime graph generated from $\Gamma$. We will call such subsets \emph{generation sites}.
\begin{definition}[Generation Site]
    \label{def:Generation Site}
    Suppose $\Gamma$ is a minimal prime graph. We say that any subset of vertices $U$ is a \emph{generation site} provided that the graph
    \begin{equation*}
        \Gamma^{\prime}=(V(\Gamma)\cup\lbrace w\rbrace, E(\Gamma)\cup\lbrace \lbrace  w, u\rbrace\mid u\in U\rbrace)
    \end{equation*}
    is a minimal prime graph generated from $\Gamma$. We call the graph $\Gamma^{\prime}$ the \emph{graph generated from $\Gamma$ at $U$}.
\end{definition}
We will use the terminology of a \emph{generation method} to indicate some criterion or feature of a graph which can be used to create a generation site. For example, we say vertex duplication is a generation method as it provides a method of finding a generation site, namely take all vertices of distance $1$ or less from a fixed vertex.

In our aim to study these generation sites, we can immediately show some basic properties of their complements.
\begin{lemma}[Complement of Generation Site is Colored by 2 Colors]
    \label{lem:GenerationSiteisColoredby2Colors}
    Suppose $\Gamma$ is a minimal prime graph and $U\subseteq V(\Gamma)$ is a generation site. Then there exists a $3$ coloring of $\overline{\Gamma}$ such that the set $K=V(\Gamma)\cut U$ is colored by at most 2 colors in the coloring.
\end{lemma}
\begin{proof}
    Let $\Gamma^{\prime}$ be the graph generated from $\Gamma$ at $U$, and let $w$ be the extra vertex in $\Gamma^{\prime}$. We see that as for all $v\in K$, $v$ is adjacent to $w$ in $\overline{\Gamma^{\prime}}$, in any coloring of $\overline{\Gamma'}$, no vertex in $K$ can be the same color as $w$. Therefore, in any $3$-coloring of $\overline{\Gamma^{\prime}}$, any $v\in K$ is one of the two colors which $w$ is not. As $\overline{\Gamma^{\prime}}$ has a $3$-coloring as $\Gamma^{\prime}$ is a minimal prime graph, by then restricting this $3$-coloring to $\overline{\Gamma}$, we find our desired $3$ coloring of $\overline{\Gamma}$.
\end{proof}
\begin{lemma}[Complement of Generation Site is a Clique]
    \label{lem:ComplementofGenerationSiteisaClique}
    Suppose $\Gamma$ is a minimal prime graph and $U\subseteq V(\Gamma)$ is a generation site. Then $K=V(\Gamma)\cut U$ is a clique in $\Gamma$.
\end{lemma}
\begin{proof}
    Let $\Gamma^{\prime}$ be the graph generated from $\Gamma$ at $U$, and let $w$ be the extra vertex in $\Gamma^{\prime}$. We now show that $K$ is a clique in $\Gamma^{\prime}$. As $w\notin K$, this is sufficient to show $K$ is a clique in $\Gamma$. If $\abs{V(K)}=1$, then $K$ is immediately a clique. Now suppose $\abs{V(K)}\ge 2$. Fix arbitrary distinct $u, v\in K$. By definition of $K$, $u$ and $v$ are not adjacent to $w$. We conclude $\lbrace w, u\rbrace$ and $\lbrace w, v\rbrace$ are edges in $\overline{\Gamma^{\prime}}$. As $\overline{\Gamma^{\prime}}$ is triangle-free since $\Gamma^{\prime}$ is a minimal prime graph, we conclude $\lbrace u, v\rbrace\not\in E(\overline{\Gamma^{\prime}})$ and thus $\lbrace u, v\rbrace\in E(\Gamma)$. We conclude that as $u$ and $v$ were arbitrary that $K$ is a clique.
\end{proof}
While one might hope the above fully characterizes generation sites, this is not the case. However, if we add the additional condition that the complement $K$ of a generation site is a maximal clique, we see we obtain a partial converse.

\begin{theorem}[Clique Generation]
    \label{thm:CliqueGeneration}
    Suppose the minimal prime graph $\Gamma=(V, E)$ contains a subset of vertices $K$ such that
    \begin{enumerate}[label=(\roman*)]
        \item $K$ is a maximal clique in $\Gamma$,
        \item There exists a coloring of $\overline{\Gamma}$ such that $K$ is colored by two colors.
    \end{enumerate}
    Then $V(\Gamma)\cut K$ is a generation site. We call the corresponding generation method \emph{clique generation}.
\end{theorem}
\begin{proof}
      Let $\Gamma^{\prime}$ be the graph generated from $\Gamma$ by $V(\Gamma)\cut K$, and let $w$ be the extra vertex in $\Gamma'$. We check $\Gamma'$ for the properties of a minimal prime graph.
      
      First, we make the simple checks $\Gamma$ is both connected and has two or more vertices
    \begin{enumerate}[label=\arabic*)]
        \begin{item}
            We see that as we added a vertex to $\Gamma$, $\Gamma'$ must have 2 or more vertices as $\Gamma$ does, as $\Gamma$ is a minimal prime graph.
        \end{item}
        \begin{item}
            We see that as $\Gamma$ is a minimal prime graph, the only possibly disconnected vertex is $w$. As $K$ is nonempty, there is at least one edge to $w$, and therefore $\Gamma'$ is connected.
        \end{item}
    \end{enumerate}
    We now check for the 3 main properties of a minimal prime graph: that the complement is triangle-free, 3-colorable, and the minimal in the minimal prime graph sense.
    \begin{enumerate}
        \begin{item}
            We see that to check $\overline{\Gamma'}$ is triangle-free, it suffices to ensure $w$ is not in a triangle. This suffices as no other triangles can exist, as such a triangle would also be in $\overline{\Gamma}$, which is triangle-free as $\Gamma$ is a minimal prime graph.

            We see that the only vertices connected to $w$ in $\overline{\Gamma'}$ are those vertices in $K$. As $K$ is a clique, no two vertices in $K$ are adjacent to each other in $\overline{\Gamma'}$. Therefore, there can be no triangle containing the vertex $w$.
        \end{item}
        \begin{item}
            We see that $\overline{\Gamma'}$ is 3-colorable by taking a coloring of $\overline{\Gamma}$ such that the vertices in $K$ are colored by two colors. As in $\overline{\Gamma'}$ the vertex $w$ is only adjacent to vertices in $K$, we can give $w$ the third color to get a valid 3-coloring of $\overline{\Gamma'}$.
        \end{item}
        \begin{item}
            We now show minimality. We see it suffices to only consider adding edges of the form $\{w, v\}$ where $u\not\in K$ to $\overline{\Gamma'}$. Suppose we add such an edge. As $K$ is a maximal clique in $\Gamma$, $u$ cannot be adjacent to all $v\in K$ in $\Gamma$. Therefore, $u$ is adjacent to at least one $v\in K$ in $\overline{\Gamma'}$. We conclude as $w$ is adjacent to all $v\in K$ in $\overline{\Gamma'}$ that $\lbrace w, u\rbrace$, $\{w, v\}$, and $\{u, v\}$ form a triangle. 

            As $\{w, v\}$ was arbitrary, we conclude any such edge introduces a triangle in $\overline{\Gamma'}$, and therefore $\Gamma'$ is a minimal prime graph.
        \end{item}\end{enumerate}\end{proof}
For minimal prime graphs of small order, their generation sites often satisfy the criteria for both vertex duplication and clique generation. For larger orders, generation sites satisfying exactly one of vertex duplication or clique generation exist, seen below.
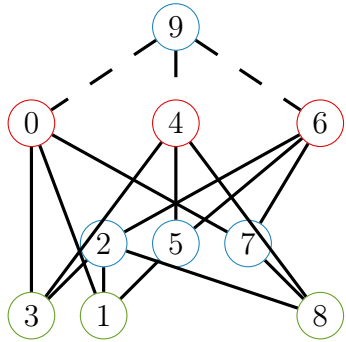
\begin{figure}[H]
    \begin{center}
\begin{tikzpicture}[scale=.32,
dot/.style = {circle,draw, inner sep=0, minimum size=1.5em},
dot/.default = 6pt  
                    ] 
\begin{scope}[every node/.style={dot}]
    \node[draw=nred] (0) at (0,8) {0};
    \node[draw=ngreen] (1) at (3,0) {1};
    \node[draw=nblue] (2) at (3,3) {2};
    \node[draw=ngreen] (3) at (0,0) {3};
    \node[draw=nred] (4) at (6,8) {4};
    \node[draw=nblue] (5) at (6,3) {5};
    \node[draw=nred] (6) at (12,8) {6};
    \node[draw=nblue] (7) at (9,3) {7};
    \node[draw=ngreen] (8) at (12,0) {8};
    \node[draw=nblue] (9) at (6,12) {9};

\end{scope}

\begin{scope}[
              every edge/.style={draw=black,very thick}]
     \path [-] (0) edge node {} (1);
\path [-] (0) edge node {} (3);
\path [-] (0) edge node {} (7);
\path [-] (1) edge node {} (2);
\path [-] (1) edge node {} (5);
\path [-] (2) edge node {} (3);
\path [-] (2) edge node {} (6);
\path [-] (2) edge node {} (8);
\path [-] (3) edge node {} (4);
\path [-] (4) edge node {} (5);
\path [-] (4) edge node {} (8);
\path [-] (5) edge node {} (6);
\path [-] (6) edge node {} (7);
\path [-] (7) edge node {} (8);
\path[dash pattern=on 10pt off 10pt] [-] (9) edge node {} (6);
\path[dash pattern=on 10pt off 10pt] [-] (9) edge node {} (4);
\path[dash pattern=on 10pt off 10pt] [-] (9) edge node {} (0);

\end{scope}
\end{tikzpicture}    
\end{center}
     \caption{\,\,Complement of a minimal prime graph demonstrating a generation site $\lbrace 1, 2, 3, 5, 7, 8\rbrace$ satisfying clique generation but not vertex duplication. Vertex 9 is not a true twin of any vertex, but is adjacent to a maximal independent set colored by at most 2 colors in the complement. This graph is available as MPG\_10\_22 in \cite{githubrepository}.}    \label{fig:vdnotcg}
     \end{figure}
\begin{figure}[H]
    \begin{center}
       \begin{tikzpicture}[scale=0.30,
dot/.style = {circle,draw, inner sep=0, minimum size=1.5em},
dot/.default = 6pt  
                    ] 
\begin{scope}[every node/.style={dot}]
\node[draw=nred] (0) at (0, 0) {0};
\node[draw=ngreen] (1) at (3.00000000000000, 4) {1};
\node[draw=nred] (2) at (4, 0) {2};
\node[draw=ngreen] (3) at (0, 8) {3};
\node[draw=nblue] (4) at (4, 12) {4};
\node[draw=nred] (5) at (4, 8) {5};
\node[draw=ngreen] (6) at (15.0000000000000, 4) {6};
\node[draw=nblue] (7) at (6, -4) {7};
\node[draw=ngreen] (8) at (8, 8) {8};
\node[draw=nred] (9) at (12, 0) {9};
\node[draw=nblue] (10) at (2, -4) {10};
\node[draw=ngreen] (11) at (12, 8) {11};
\node[draw=nred] (12) at (16, 0) {12};
\node[draw=nblue] (13) at (10, -4) {13};
\node[draw=nblue] (14) at (16, 8) {14};
\node[draw=nblue] (15) at (8, 12) {15};

\end{scope}

\begin{scope}[
              every edge/.style={draw=black,very thick}]
\path [-] (0) edge node {} (1);
\path [-] (0) edge node {} (3);
\path [-] (0) edge node {} (7);
\path [-] (0) edge node {} (11);
\path [-] (0) edge node {} (13);
\path [-] (0) edge node {} (14);
\path [-] (1) edge node {} (2);
\path [-] (1) edge node {} (5);
\path [-] (1) edge node {} (9);
\path [-] (1) edge node {} (12);
\path [-] (2) edge node {} (3);
\path [-] (2) edge node {} (6);
\path [-] (2) edge node {} (8);
\path [-] (2) edge node {} (11);
\path [-] (3) edge node {} (4);
\path [-] (3) edge node {} (10);
\path [-] (4) edge node {} (5);
\path [-] (4) edge node {} (8);
\path [-] (4) edge node {} (11);
\path [-] (5) edge node {} (6);
\path [-] (5) edge node {} (13);
\path [-] (6) edge node {} (7);
\path [-] (6) edge node {} (10);
\path [-] (6) edge node {} (14);
\path [-] (7) edge node {} (8);
\path [-] (7) edge node {} (12);
\path [-] (8) edge node {} (9);
\path [-] (8) edge node {} (13);
\path [-] (9) edge node {} (10);
\path [-] (9) edge node {} (11);
\path [-] (9) edge node {} (14);
\path [-] (10) edge node {} (12);
\path [-] (11) edge node {} (12);
\path [-] (12) edge node {} (13);
\path [-] (12) edge node {} (14);

\path [dash pattern=on 10pt off 10pt] (15) edge node {} (3);
\path [dash pattern=on 10pt off 10pt] (15) edge node {} (5);
\path [dash pattern=on 10pt off 10pt] (15) edge node {} (8);
\path [dash pattern=on 10pt off 10pt] (15) edge node {} (11);
\end{scope}
\end{tikzpicture}    
\end{center}
    \caption{\,\,Complement of a minimal prime graph demonstrating a generation site $\lbrace 0, 1, 2, 4, 6, 7, 9, 10, 12, 13\rbrace$ satisfying vertex duplication but not clique generation. Vertex 15 is a true twin of vertex 4, but is not adjacent to a maximal independent set. This graph is available as MPG\_15\_2327 in \cite{githubrepository}.}    \label{fig:cgnotvd}
\end{figure}
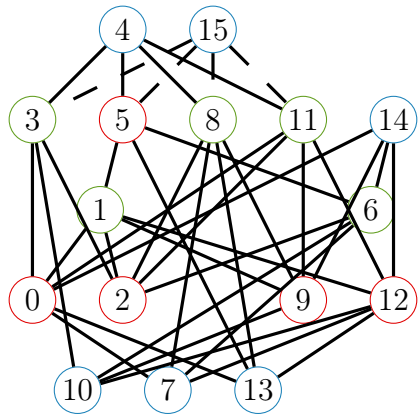

While we have shown the existence of other generation methods than vertex duplication, there are still many questions. For one, the current formulation of clique generation requires a choice of coloring. While we expect generation methods for minimal prime graphs to involve criteria related to coloring, given the importance of coloring in the definition of a minimal prime graph, we hope there is an alternative formulation or generalization which can avoid requiring this choice. Furthermore, as we show in , minimal prime graphs generated through vertex duplication share similar structural properties with the minimal prime graph they were generated from. While such a relation is less clear for more general generation methods, it is possible the additional structure of being a minimal prime graph will allow for some structure to be preserved.

We note that as there exist minimal prime graphs with generation sites which satisfy neither vertex duplication nor clique generation that other generation methods exist. For an example, see graph MPG\_16\_7432 of \cite{githubrepository}, which has a generating site $\{1, 2, 3, 4, 5, 7, 8, 10, 11, 13, 14\}$ that does not satisfy the conditions of vertex duplication or clique generation. However, we were unable to find any unifying property for the other generation sites we found.

Finally, a natural question is to consider what happens to the definition of base graph with this broader notion of generating a minimal prime graph. For example, there are base graphs which are generated from minimal prime graphs through clique generation. 
As such, one might consider the notion of a ``super'' base graph, which is a minimal prime graph which is not generated from any minimal prime graph. Computation for $n\le 100$ has shown the graphs $G_{n,k}$ with $n,k$ as in \Cref{thm:bsegraph} are super base graphs, and we conjecture this holds for all $n,k$ for which $G_{n,k}$ is a minimal prime graph. 
\begin{conjecture}
    Let $n\ge 5$ with $n\equiv 0, 5\mod 6$ and $k=\lfloor (n+2)/6\rfloor$. Then $\overline{G_{n,k}}\cut v$ for any vertex $v$ is not a minimal prime graph, and so these graphs cannot be generated from any minimal prime graph.
\end{conjecture}
\section{Acknowledgement}
    This research was conducted at Texas State University under NSF-REU grant DMS-1757233 and NSA grant H98230-21-1-0333 during the summer of 2021. The authors thank NSF and NSA for the financial support. The first, third, fourth, and fifth authors thank Texas State University for running the REU online during this difficult period of social distancing and providing a welcoming and supportive work environment. Those authors also thank their mentor, the second author Dr. Thomas Michael Keller, for his invaluable advice and guidance throughout this project. Professor Yong Yang, the director of the REU program, is recognized for conducting an inspiring and successful research program. 
    
\printbibliography

@article{gruber,
title = {A characterization of the prime graphs of solvable groups},
journal = {J. Algebra},
volume = {442},
pages = {397-422},
year = {2015},
author = {Alexander Gruber and Thomas Michael Keller and Mark L. Lewis and Keeley Naughton and Benjamin Strasser},
}

@misc{florez,
  doi = {10.48550/ARXIV.2011.10403},
  url = {https://arxiv.org/abs/2011.10403},
  author = {Florez, Chris and Higgins, Jonathan and Huang, Kyle and Keller, Thomas Michael and Shen, Dawei},
  title = {Minimal Prime Graphs of Solvable Groups},
  publisher = {arXiv},
  year = {2020, submitted},
}

@article{bipartitegraphtensor,
title = {Imbeddings of the tensor product of graphs where the second factor is a complete graph},
journal = {Discrete Math.},
volume = {182},
number = {1},
pages = {13-19},
year = {1998},
author = {Ghidewon {Abay Asmerom}},
}

@article{sidorenko,
title = {Triangle-free regular graphs},
journal = {Discrete Math.},
volume = {91},
number = {2},
pages = {215-217},
year = {1991},
author = {A.F. Sidorenko},
}

@misc{githubrepository,
  url = {https://github.com/plotnw/minimal-prime-graphs},
  author = {Wen Plotnick},
  title = {Minimal Prime Graph Repository},
  publisher = {GitHub},
  year = {2022},
  commit = {68449e9a3b18958ba6a3526b9b00429671ac616e},
}

@manual{sage,
  Key          = {SageMath},
  Author       = {{The Sage Developers}},
  Title        = {{S}ageMath, the {S}age {M}athematics {S}oftware {S}ystem ({V}ersion 9.6)},
  note         = {{\tt https://www.sagemath.org}},
  Year         = {2022},
}

\end{document}